\documentclass{article}%
\usepackage{graphicx}
\usepackage{amsmath}
\usepackage{amsfonts}
\usepackage{amssymb}%
\providecommand{\U}[1]{\protect\rule{.1in}{.1in}}
\newtheorem{theorem}{Theorem}
{}

\newtheorem{case}{Case}

\newtheorem{corollary}{Corollary}

\newtheorem{lemma}{Lemma}
{}

\newtheorem{proposition}{Proposition}
\newtheorem{remark}{Remark}

\newenvironment{proof}[1][Proof]{\textbf{#1.} }{\ \rule{0.5em}{0.5em}}

\begin{document}

\title{On the real spectrum of differential operators with PT-symmetric periodic
matrix coefficients}
\author{O. A. Veliev\\{\small Dogus University, \ Istanbul, Turkey.}\\\ {\small e-mail: oveliev@dogus.edu.tr}}
\date{}
\maketitle

\begin{abstract}
We study the spectrum of the differential operator $T$ generated by the
differential expression of order $n>2$ with the $m\times m$ PT-symmetric
periodic matrix coefficients. The case when $m$ and $n$ are the odd numbers
was investigated in [8]. In this paper, we consider the all remained cases:
\textbf{(a)} $n$ is an odd number and $m$ is an even number, \textbf{(b)} $n$
is an even number and $m$ is an arbitrary positive integer. We find conditions
on the coefficients under which in the cases \textbf{(a)} and \textbf{(b)} the
spectrum of $T$ contains the sets $(-\infty,H]$ $\cup\lbrack H,\infty)$ and
$[H,\infty)$ respectively for some $H>0.$

Key Words: Non-self-adjoint differential operator, PT-symmetric periodic
matrix coefficients. Real spectrum.

AMS Mathematics Subject Classification: 34L05, 34L20.

\end{abstract}

\section{Introduction and Preliminary Facts}

In this paper, we consider the spectrum $\sigma(T)$ of the differential
operator $T$ generated in the space $L_{2}^{m}(-\infty,\infty)$ by the
differential expression
\begin{equation}
(-i)^{n}y^{(n)}+(-i)^{n-2}P_{2}y^{(n-2)}+(-i)^{n-3}P_{3}y^{(n-3)}+...+P_{n}y,
\tag{1}%
\end{equation}
where $n>2,$ $P_{k}=(p_{k,i,j})$ for $k=2,3,...,n$ are the $m\times m$
matrices with the complex-valued PT-symmetric periodic entries
\begin{equation}
p_{k,i,j}\left(  x+1\right)  =p_{k,i,j}\left(  x\right)  ,\text{ }%
p_{k,i,j}\left(  -x\right)  =\overline{p_{k,i,j}\left(  x\right)  },\text{
}p_{k,i,j}\in L_{2}[0,1] \tag{2}%
\end{equation}
and $y=(y_{1},y_{2},...,y_{m})^{T}$ is a vector-valued function. Here
$L_{2}^{m}(a,b)$ for $-\infty\leq a<b\leq\infty$ is the space of the
vector-valued functions $f=\left(  f_{1},f_{2},...,f_{m}\right)  ^{T}$ with
the norm $\left\Vert \cdot\right\Vert _{(a,b)}$ and inner product
$(\cdot,\cdot)_{(a,b)}$ defined by%
\[
\left\Vert f\right\Vert _{(a,b)}^{2}=\int_{a}^{b}\left\vert f\left(  x\right)
\right\vert ^{2}dx,\text{ }(f,g)_{(a,b)}=\int_{a}^{b}\left\langle f\left(
x\right)  ,g\left(  x\right)  \right\rangle dx,
\]
where $\left\vert \cdot\right\vert $ and $\left\langle \cdot,\cdot
\right\rangle $ are the norm and inner product in $\mathbb{C}^{m}.$

It is well known that (see for example [3, 5]) the spectrum $\sigma(T)$ of $T$
is the union of the spectra $\sigma(T_{t})$ of the operators $T_{t}$ for
$t\in\lbrack0,2\pi)$ generated in $L_{2}^{m}[0,1]$ by (1) and the boundary
conditions
\begin{equation}
y^{(\mathbb{\nu})}\left(  1\right)  =e^{it}y^{(\mathbb{\nu})}\left(  0\right)
,\text{ }\mathbb{\nu}=0,1,...,(n-1). \tag{3}%
\end{equation}
The spectrum of $T_{t}$ consists of the eigenvalues. These eigenvalues are
known as Bloch eigenvalues of $T$ and are the roots of the characteristic
equation
\begin{equation}
\Delta(\lambda,t):=\det(Y_{j}^{(\nu-1)}(1,\lambda)-e^{it}Y_{j}^{(\nu
-1)}(0,\lambda))_{j,\nu=1}^{n}= \tag{4}%
\end{equation}%
\[
e^{inmt}+f_{1}(\lambda)e^{i(nm-1)t}+f_{2}(\lambda)e^{i(nm-2)t}+\cdot\cdot
\cdot+f_{nm-1}(\lambda)e^{it}+1,
\]
where $f_{1}(\lambda),f_{2}(\lambda),...$ are the entire functions,
$Y_{1}(x,\lambda),Y_{2}(x,\lambda),\ldots,Y_{n}(x,\lambda)$ are the solutions
of the matrix equation
\[
(-i)^{n}Y^{(n)}+(-i)^{n-2}P_{2}Y^{(n-2)}+(-i)^{n-3}P_{3}Y^{(n-3)}%
+...+P_{n}Y=\lambda Y
\]
satisfying $Y_{k}^{(j)}(0,\lambda)=O_{m}$ for $j\neq k-1$, $Y_{k}%
^{(k-1)}(0,\lambda)=I_{m}$. Here $O_{m}$ and $I_{m}$ are the $m\times m$ zero
and identity matrices (see [4, Chapter 3]).

Note that there are a large number of papers for the scalar case $m=1$ and
$n=2$, namely for the Schr\"{o}dinger operator (see the monographs [1,
Chapters 4 and 6] and [6, Chapters 3 and 5] and the papers they refer to). The
results and the method used in this paper are completely different from the
results and methods of those papers. Therefore we do not discuss the scalar
case in detail.

As far as I know, only the papers [7, 8] were devoted to the differential
operator with the periodic PT-symmetric matrix coefficients. In [7] the
Schr\"{o}dinger operator with a PT-symmetric periodic matrix potential was
investigated, where $n=2$. In [8] we considered the following case.

\begin{case}
$m$ and $n$ are the odd numbers.
\end{case}

We proved that in Case 1, $\sigma(T)$ contains all real line $\mathbb{R}.$ In
this paper, we consider the others and all the remained cases, namely the
following cases:

\begin{case}
$n$ is an odd number and $m$ is an even number.
\end{case}

\begin{case}
$n$ is an even number and $m$ is an arbitrary positive integer.
\end{case}

Therefore, this paper can be considered as a continuation and completion of
the paper [8]. Moreover, the method used in [8] for Case 1 can note be used
for Cases 2 and 3, since the method of Case 1 passes through only if $nm$ is
an odd number. That is why, the methods used in [8] and in this paper are
completely different.

The paper is organized as follows. To study the spectrum of $T_{t},$ we
consider the family of the operators $T_{t}(\varepsilon,C)$ generated by the
differential expression
\begin{equation}
(-i)^{n}y^{n}+(-i)^{n-2}Cy^{(n-2)}+\varepsilon\left(  (-i)^{n-2}%
(P_{2}-C)y^{(n-2)}+%
{\textstyle\sum\limits_{l=3}^{n}}
(-i)^{n-l}P_{l}y^{(n-l)}\right)  \tag{5}%
\end{equation}
and boundary conditions (3), where $\varepsilon\in\lbrack0,1],$ $C=\int
_{0}^{1}P_{2}\left(  x\right)  dx,$ $T_{t}(1,C)=T_{t}$ and $T_{t}%
(0,C)=:T_{t}(C)$ is the operator generated by the expression
\[
(-i)^{n}y^{(n)}+(-i)^{n-2}Cy^{(n-2)}%
\]
and boundary conditions (3). Thus $T_{t}(C)$ and $T_{t}(\varepsilon
,C)-T_{t}(C)$ are respectively the unperturbed operator and perturbation. We
prove that the large eigenvalues of $T_{t}(\varepsilon,C)$ are located in the
small neighborhood of the eigenvalues of $T_{t}(C).$ Therefore, first of all,
let us analyze the eigenvalues and eigenfunction of the operator $T_{t}(C)$.
Using (2) one can easily verify that the entries of the matrix $C$ are the
real numbers. Therefore, the eigenvalues of the matrix $C$ consist of the real
eigenvalues and the pairs of the conjugate complex numbers. The distinct
eigenvalues of $C$ are denoted by $\mu_{1},\mu_{2},...,\mu_{p}.$ If the
multiplicity of $\mu_{j}$ is $m_{j},$ then
\begin{equation}
m_{1}+m_{2}+...+m_{p}=m. \tag{6}%
\end{equation}
Without loss of generality, we denote the real eigenvalues by $\mu_{1}<\mu
_{2}<...<\mu_{s}$ and the nonreal eigenvalues by$\ \mu_{s+1},\mu_{s+2}%
,...,\mu_{p}.$ One can easily verify that the eigenvalues and eigenfunctions
of $T_{t}(C)$ are respectively $\ $%
\begin{equation}
\mu_{k,j}(t)=\left(  2\pi k+t\right)  ^{n}+\mu_{j}\left(  2\pi k+t\right)
^{n-2} \tag{7}%
\end{equation}
and
\begin{equation}
\Phi_{k,j,l,t}(x)=u_{j,l}e^{i\left(  2\pi k+t\right)  x} \tag{8}%
\end{equation}
for $k\in\mathbb{Z},$ $j=1,2,...,p$ and $l=1,2,...,l_{j\text{ }},$ where
$u_{j,1},$ $u_{j,2},...u_{j,l_{j}}$ are the linearly independent eigenvectors
corresponding to the eigenvalue $\mu_{j}$ of $C.$ Therefore $\sigma\left(
T(C)\right)  $ consists of the lines and half lines, respectively if $n$ is an
odd and even number. To see the difference of the investigations in Cases 2
and 3 let us discuss the exceptional points of $\sigma(T(C)),$ since the
perturbation $T(\varepsilon,C)-T(C)$ for the small values of $\varepsilon$ may
generate the gaps in $\sigma\left(  T(C)\right)  $ only at the neighborhoods
of the exceptional Bloch eigenvalues. Note that the exceptional points of
$\sigma\left(  T(C)\right)  $ are the points $\mu_{k,j}(t_{0})\in\sigma\left(
T(C)\right)  ,$ where the multiplicity of the eigenvalues $\mu_{k,j}(t)$
varies in any neighborhood of $t_{0}$. In Proposition 1$(a)$, we prove that if
$n$ is an odd number, then the multiplicity of the eigenvalues $\mu_{k,j}(t)$
is equal to $m_{j}$ for all $t\in\mathbb{R}$. Therefore, in Case 2
$\sigma\left(  T(C)\right)  $ has no exceptional points. This situation
simplifies the study of Case 2. In this case we prove that if the matrix $C$
has at least one real eigenvalue of odd multiplicity, then $\sigma(T)$
contains the set $(-\infty,-H]\cup\lbrack H,\infty)$ for some $H\geq0$.
However, if $n$ is an even number, then there exist the points $t\in
\lbrack0,2\pi)$ (see Proposition 1$(b)$) such that $\mu_{k,j}(t)=\mu_{l,i}(t)$
for some $(l,i)\neq(k,j),$ that is, the multiplicity of the eigenvalues
$\mu_{k,j}(t)$ varies. Therefore, in this case, $\sigma\left(  T(C)\right)  $
may have infinitely many exceptional points. This situation complicate the
investigation of Case 3. In this case we prove that if the matrix $C$ has at
least three real eigenvalues of odd multiplicity satisfying some conditions
(see (27)), then $\sigma(T)$ contains the set $[H,\infty)$ for some $H\geq0$.
Fortunately, the investigations [7] of the case $n=2$ helps us to consider
this complicated case.

Note that we use the following theorem, which can be proved by repeating the
proof of the Theorem 1 of [7].

\begin{theorem}
If $\lambda$ is an eigenvalue of multiplicity $v$ of the operator $T_{t},$
then $\overline{\lambda}$ is also an eigenvalue of the same multiplicity of
$T_{t}$
\end{theorem}

To formulate the following theorem, which is essentially used in this paper,
we introduce the following notation. Denote by $u_{j,l,1},$ $u_{j,l,2}%
,...u_{j,l,r_{j,l}-1}$ the associated vectors corresponding to the eigenvector
$u_{j,l},$ such that $\left(  C-\mu_{j}I\right)  u_{j,l,q}=u_{j,l,q-1}$ for
$q=1,2,...,r_{j,l}-1,$ where $u_{j,l,0}=u_{j,l}$ and recall that $u_{j,1},$
$u_{j,2},...u_{j,l_{j}}$ are the linearly independent eigenvectors
corresponding to the eigenvalue $\mu_{j}$ of $C$. Then it is not hard to
verify that
\[
\left(  T_{t}(C)-\mu_{k,j}(t)I\right)  u_{j,l,q}e^{i\left(  2\pi k+t\right)
x}=u_{j,l,q-1}e^{i\left(  2\pi k+t\right)  x}%
\]
and
\begin{equation}
r_{j,1}+r_{j,2}+...+r_{j,l_{j}}=m_{j}, \tag{9}%
\end{equation}
where $m_{j}$ is the multiplicity of the eigenvalue $\mu_{j}.$ It means that
the associated functions of $T_{t}(C)$ corresponding to the eigenfunction
$\Phi_{k,j,l,t}(x)=u_{j,l}e^{i\left(  2\pi k+t\right)  x}$ are%
\begin{equation}
\Phi_{k,j,l,q,t}(x)=u_{j,l,q}e^{i\left(  2\pi k+t\right)  x} \tag{10}%
\end{equation}
for $q=1,2,...,r_{j,l}-1.$ Thus the dimension of the space generated by the
functions (8) and (10) is the multiplicity $m_{j}$ of the eigenvalue $\mu
_{j},$ due to (9).

\begin{theorem}
\textit{ There exist positive numbers }$N$ and $c$ such that \textit{the large
eigenvalues of }$T_{t}(\varepsilon,C)$ \textit{lie in the disks }%
\[
U_{\varepsilon_{k}}(\mu_{k,j}(t)):=\left\{  \lambda\in\mathbb{C}:\left\vert
\lambda-\mu_{k,j}(t)\right\vert <\varepsilon_{k}\right\}
\]
\textit{for }$\left\vert k\right\vert \geq N$ and $j=1,2,...,p,$\textit{ where
}$\varepsilon_{k}=c\left(  \mid k\mid^{n-3}\right)  ^{1/r}$ if $n$ is an odd
number,\textit{ }%
\[
\varepsilon_{k}=c\left(  \left(  \mid k^{-1}\mid+q_{k}\right)  \left\vert
k\right\vert ^{n-2}\right)  ^{1/r},
\]
if $n$ is an even number, $r=\max\limits_{j=1,2,...,p}\left\{  r_{j,1}%
,r_{j,2},...,r_{j,l_{j}}\right\}  ,$ $t\in\lbrack-1,2\pi-1)$\textit{,
}$\varepsilon\in\lbrack0,1]$, \textit{ }%
\[
q_{k}=\max\left\{  \left\vert p_{2,i,j,l}\right\vert :i,j=1,2,...,m;\text{
}l=\pm2k,\pm(2k+1)\right\}
\]
and $p_{2,i,j,l}=\int\nolimits_{[0,1]}p_{2,i,j}(x)e^{-2\pi ilx}dx.$
\end{theorem}

Theorems analogous to Theorem 2 in the cases: \textbf{(i)} $n=2$ and
\textbf{(ii)} $T_{t}$ is a self-adjoint operator were proved in [7] and [9],
respectively. Since these cases do not cover the operator $T_{t}$, we cannot
directly refer to these papers. However, the proof of this theorem is similar
to the proofs of the corresponding Theorems 3 and 5 in [7] and [9],
respectively. Therefore, we present the proof of Theorem 2 in the Appendix.

\section{Main Results}

First, let us consider the eigenvalues and root functions of $T_{t}(C)$ for
$t\in\lbrack0,2\pi).$

\begin{proposition}
Let $\mu_{j}$ be an eigenvalue of $C$ of multiplicity $m_{j}.$

$(a)$ If $n$ is an odd number, then the multiplicity of the eigenvalue
$\mu_{k,j}(t)$ of $T_{t}(C)$ is $m_{j}$ for all $k\in\mathbb{Z}$ and
$t\in\lbrack0,2\pi),$ where $\mu_{k,j}(t)$ is defined in (7).

$(b)$ If $n$ is an even number, then the multiplicity of $\mu_{k,j}(t)$ is
$m_{j}$ for $t\in\lbrack0,2\pi)\backslash A(k,j),$ where
\begin{equation}
A(k,j)=%
{\textstyle\bigcup\limits_{l\in\mathbb{Z},i=1,2,...,p}}
\left\{  t_{l,i,q}:q=1,2,...,n\right\}  \tag{11}%
\end{equation}
and $t_{l,i,1},$ $t_{l,i,2},...,t_{l,i,n}$ are the roots of the equation%
\[
\left(  2\pi k+t\right)  ^{n}+\mu_{j}\left(  2\pi k+t\right)  ^{n-2}=\left(
2\pi l+t\right)  ^{n}+\mu_{i}\left(  2\pi l+t\right)  ^{n-2}.
\]

\end{proposition}

\begin{proof}
It follows from (8)-(10) that, if
\begin{equation}
\mu_{k,j}(t)\neq\mu_{l,i}(t) \tag{12}%
\end{equation}
for $(l,i)\neq(k,j),$ then $\mu_{k,j}(t)$ is an eigenvalue of $T_{t}(C)$ of
the multiplicity $m_{j}.$ On the other hand, by (7), if $n$ is an odd and even
number respectively, then (12) holds for all $t\in\lbrack0,2\pi)$ and
$t\in\lbrack0,2\pi)\backslash A(k,j).$ Therefore, the proposition is true.
\end{proof}

Thus the spectrum $\sigma\left(  T(C)\right)  $ has no exceptional points if
$n$ is an odd number, while the spectrum $\sigma\left(  T(C)\right)  $ has
infinitely many exceptional points if $n$ is an even number. Moreover, the set
of $t\in\lbrack0,2\pi)$ for which $\mu_{k,j}(t)$ are the exceptional points
has the accumulation points $0,\pi$ and $2\pi.$ Since $T_{t}=T_{t+2\pi},$
sometimes, instead of $t\in\lbrack0,2\pi)$ we use $t\in\lbrack-h,2\pi-h)$ for
some $h\in(0,\pi)$ in order to get two accumulation points. Note that in
Theorem 2 and Proposition 1 one can replace $[-1,2\pi-1)$ and $[0,2\pi)$ by
$[-h,2\pi-h).$

Now consider the large eigenvalues of $T_{t}$, by using Proposition 1, Theorem
2 and the notation $a_{k}\asymp b_{k}$ which means that there exist constants
$c_{1},$ $c_{2}$ and $c_{3},$ independent of $t$ and $\varepsilon,$ such that
$c_{1}|a_{k}|<\left\vert b_{k}\right\vert <c_{2}|a_{k}|$ for all $\left\vert
k\right\vert >c_{3}.$ Note that in the forthcoming inequalities we denote by
$c_{1},$ $c_{2},...$ positive constants independent of $t$ and $\varepsilon.$

\begin{theorem}
Let $\mu_{j}$ be an eigenvalue of $C$ of multiplicity $m_{j}.$

$(a)$ If $n$ is an odd number, then the operator $T_{t}$ has only $m_{j}$
eigenvalues lying in $U_{\varepsilon_{k}}(\mu_{k,j}(t))$ \textit{for
}$\left\vert k\right\vert \geq N$ and $t\in\lbrack-h,2\pi-h),$ where $N$ and
$\varepsilon_{k}$ are defined in Theorem 2 and $h\in(0,\pi).$

$(b)$ If $n$ is an even number, then there exists $\delta_{k}\asymp\left(
\varepsilon_{k}+\varepsilon_{-k}+\varepsilon_{-k-1}\right)  k^{1-n}$ such that
the operator $T_{t}$ has only $m_{j}$ eigenvalues lying in $U_{\varepsilon
_{k}}(\mu_{k,j}(t))$ \textit{for }$\left\vert k\right\vert \geq N$ and
\begin{equation}
t\in\lbrack-h,2\pi-h)\backslash U_{\delta_{k}}\left(  A(k,j)\right)  ,\tag{13}%
\end{equation}
where $U_{\delta_{k}}\left(  E\right)  $ denotes the open $\delta_{k}$
neighborhood of the set $E$ and $A(k,j)$ is defined in (11).
\end{theorem}

\begin{proof}
$(a)$ Using (7) one can easily verify that if $\left\vert k\right\vert \geq
N$, then there exists a constant $c_{1}>0$ such that
\begin{equation}
\left\vert \mu_{k,j}(t)-\mu_{l,i}(t)\right\vert \geq c_{1}\left\vert
k\right\vert ^{n-2} \tag{14}%
\end{equation}
for all $(l,i)\neq(k,j).$ On the other hand $\varepsilon_{k}=o(k^{n-2}).$
Therefore, from Theorem 2 we obtain that the circle $D\left(  \mu
_{k,j},\varepsilon_{k}\right)  =\left\{  \lambda\in\mathbb{C}:\left\vert
\lambda-\mu_{k,j}\right\vert =\varepsilon_{k}\right\}  $ belong to the
resolvent set of the operators $T_{t}(\varepsilon,C)$ for all $\varepsilon
\in\lbrack0,1],$ where $T_{t}(\varepsilon,C)$ is generated by (5). This
implies that the operators $T_{t}:=T_{t}(1,C)$ and $T_{t}(C):=T_{t}(0,C)$ have
the same number of eigenvalues (counting the multiplicity) inside $D\left(
\mu_{k,j},\varepsilon_{k}\right)  ,$ since $T_{t}(\varepsilon,C)$ is the
halomorphic family (with respect to $\varepsilon,$ in the sense of [2] (see
[2, Chapter 7])) of operators. Thus the proof of this theorem follows from
Proposition 1$(a),$ because the operator $T_{t}(C)$ has only one eigenvalue
$\mu_{k,j}$ inside $D\left(  \mu_{k,j},\varepsilon_{k}\right)  $ and the
multiplicity of $\mu_{k,j}$ is $m_{j}.$

$(b)$ First we prove that if (13) holds, then
\begin{equation}
D\left(  \mu_{k,j}(t),\varepsilon_{k}\right)  \cap D\left(  \mu_{l,i}%
(t),\varepsilon_{l}\right)  =\varnothing\tag{15}%
\end{equation}
for $(l,i)\neq(k,j),$ where $\left\vert k\right\vert \geq N.$ It follows from
(7) that if $t\in\lbrack-h,2\pi-h),$ then $\mu_{k,j}(t)-\mu_{k,i}(t)\asymp
k^{n-2}$ for $j\neq i$ and $\left\vert \mu_{k,j}(t)-\mu_{l,i}(t)\right\vert
>d_{k}$ for $l\neq k,-k,-(k+1),$ where $d_{k}\asymp k^{n-1}.$ Thus, (15) holds
for $t\in\lbrack-h,2\pi-h)$ and $l\neq k,-k,-(k+1).$

The validity of (15) for the case $l=k$ follows from (14) and the definition
$\varepsilon_{k}.$ To prove (15) for the cases $l=-k$ and $l=-(k+1),$ let us
consider the functions $f(t)=\mu_{k,j}(t)-\mu_{-k,i}(t),$ $g(t)=\mu
_{k,j}(t)-\mu_{-k-1,i}(t).$ Using (7) and the binomial expansion of
$(a+b)^{n}$ we obtain
\[
f(t)=(2\pi k)^{n-2}(4nk\pi t+\mu_{j}-\mu_{i})+O(k^{n-3}),\text{ }f\left(
\frac{\mu_{i}-\mu_{j}}{4nk\pi}\right)  =O(k^{n-3}).
\]
On the other hand, one can easily verify that $f^{^{\prime}}(t)\asymp
k^{n-1}.$ Therefore, there exists $\delta_{k}\asymp\left(  \varepsilon
_{k}+\varepsilon_{-k}+\varepsilon_{-k-1}\right)  k^{1-n}$ such that if $t$
does not belong to the interval
\begin{equation}
U(i,j,k,\delta_{k})=\left(  \frac{\mu_{i}-\mu_{j}}{4nk\pi}-\delta_{k}%
,\frac{\mu_{i}-\mu_{j}}{4nk\pi}+\delta_{k}\right)  , \tag{16}%
\end{equation}
then $\left\vert f(t)\right\vert >\varepsilon_{k}+\varepsilon_{-k}$. In the
same way we prove that if $t$ does not belong to the interval
\begin{equation}
U(i,j,-k-1,\delta_{k})=\left(  \pi+\frac{\mu_{i}-\mu_{j}}{2\pi n(2k+n-1)}%
-\delta_{k},\pi+\frac{\mu_{i}-\mu_{j}}{2\pi n(2k+n-1)}+\delta_{k}\right)  ,
\tag{17}%
\end{equation}
then $\left\vert g(t)\right\vert >\varepsilon_{k}+\varepsilon_{-k-1}.$ Thus,
using (11) we obtain that
\begin{equation}
\left(  U_{\delta_{k}}\left(  A(k,j)\right)  \cap\lbrack-1,2\pi-1)\right)
\subset\left(
{\textstyle\bigcup\limits_{i=1}^{s}}
\left(  U(i,j,k,\delta_{k})\cup U(i,j,-k-1,\delta_{k})\right)  \right)  .
\tag{18}%
\end{equation}
Therefore, (15) is true if (13) holds.

Now, (15) with Theorem 2 implies that the circle $D\left(  \mu_{k,j}%
(t),\varepsilon_{k}\right)  $ belong to the resolvent set of the operators
$T_{t}(\varepsilon,C)$ for all $\varepsilon\in\lbrack0,1],$ if (13) holds.
Therefore instead of Proposition 1$(a)$ using Proposition 1$(b)$ and repeating
the last statements of the proof of $(a)$, we get the proof of $(b).$
\end{proof}

Now, using Theorem 3 and the following arguments we consider the gaps in
$\sigma(T)$. The substitution $y(x)=e^{itx}\widetilde{y}(x)$ implies that the
operator $T_{t}$ is generated by the differential operation
\[
(-i)^{n}\left(  \frac{\partial}{\partial x}+it\right)  ^{n}+(-i)^{n-2}%
P_{2}\left(  \frac{\partial}{\partial x}+it\right)  ^{n-2}+...+P_{n}y
\]
and the periodic boundary conditions. Then the domain of the definition of
$T_{t}$ does not depend on $t$ and hence $\{T_{t}:$ $t\in\lbrack-h,2\pi-h)\}$
is a halomorphic family (in the sense of [2] (see [2, Chapter 7])) of
operators with compact resolvent for each $h\in\lbrack0,\pi]$.

It follows from Theorem 2 and the proof of Theorem 3 that, if $\left\vert
k\right\vert \geq N$, $t_{0}\in\lbrack-h,2\pi-h)$ and $t_{0}\in\lbrack
-h,2\pi-h)\backslash U_{\delta_{k}}\left(  A(k,j)\right)  $ respectively for
odd and even $n,$ then the circle$\ D_{\varepsilon_{k}}(\mu_{k,j}(t_{0}))$
belong to the resolvent set of the operator $T_{t_{0}}$. It means that
$\Delta(\lambda,t_{0})\neq0$ for each $\lambda\in D_{\varepsilon_{k}}%
(\mu_{k,j}(t_{0})),$ where $\Delta(\lambda,t)$ is defined in (4). Since
$\Delta(\lambda,t_{0})$ is a continuous function on the compact
$D_{\varepsilon_{k}}(\mu_{k,j}(t_{0})),$ there exists $a>0$ such that
$\left\vert \Delta(\lambda,t_{0})\right\vert >a$ for all $\lambda\in
D_{\varepsilon_{k}}(\mu_{k,j}(t_{0})).$ Moreover, by (4), $\Delta(\lambda,t)$
is a polynomial of $e^{it}$ with entire coefficients. Therefore, there exists
$\delta$ such that $\left\vert \Delta(\lambda,t)\right\vert >a/2$ for all
$t\in(t_{0}-\delta,t_{0}+\delta)$ and $\lambda\in D_{\varepsilon_{k}}%
(\mu_{k,j}(t_{0})).$ It means that $D_{\varepsilon_{k}}(\mu_{k,j}(t_{0}))$
belong to the resolvent set of $T_{t}$ for all $t\in(t_{0}-\delta,t_{0}%
+\delta).$ Hence, the spectrum of $T_{t}$ is separated by $D_{\varepsilon_{k}%
}(\mu_{k,j}(t_{0}))$ into two parts in the sense of [2] (see [2, Chapter 3,
Section 6.4]). Therefore, the theory of halomorphic family of the finite
dimensional operators [2, Chapter 2] can be applied to the part of $T_{t}$ for
$t\in(t_{0}-\delta,t_{0}+\delta)$ corresponding to the inside of
$D_{\varepsilon_{k}}(\mu_{k,j}(t_{0})).$ Now, using these arguments we prove
the following lemma which plays the crucial role in the prove of the main
results of this paper.

\begin{lemma}
Suppose that the matrix $C$ has a real eigenvalue $\mu_{j}$ of odd
multiplicity $m_{j}$. If $\lambda=\mu_{k,j}(t_{0})$\textit{ for some }%
$t_{0}\in\lbrack-1,2\pi-1)$ and the disk $U_{\varepsilon_{k}}(\mu_{k,j}(t))$
for all $t\in\left[  t_{0}-h_{k},t_{0}+h_{k}\right]  $ contains only $m_{j}$
eigenvalues (counting the multiplicity) of $T_{t},$ then the spectrum of $T$
contains the point $\lambda,$ where $k\geq N$, $h_{k}=2\varepsilon_{k}%
(2\pi(k-1))^{1-n},$ $\varepsilon_{k}$ and $N$ are defined in Theorem 2.
\end{lemma}

\begin{proof}
It follows from (7) that
\begin{equation}
(2\pi(k-1))^{n-1}\leq\left\vert \frac{d\mu_{k,j}(t)}{dt}\right\vert \leq
(2\pi(k+2))^{n-1} \tag{19}%
\end{equation}
for all $t\in\lbrack-\pi,2\pi).$ Therefore%
\begin{equation}
\mu_{k,j}(t_{0}-h_{k})\leq\lambda-2\varepsilon_{k},\text{ }\mu_{k,j}%
(t_{0}+h_{k})\geq\lambda+2\varepsilon_{k}. \tag{20}%
\end{equation}
Denote by $\lambda_{k,1}(t),\lambda_{k,2}(t),...,\lambda_{k,m_{j}}(t)$ the
eigenvalues of $T_{t}$ lying in $U_{\varepsilon_{k}}(\mu_{k,j}(t))$ and
consider the unordered $m_{j}$-tuple $\Omega(t):=\left\{  \lambda
_{k,1}(t),\lambda_{k,2}(t),...,\lambda_{k,m_{j}}(t)\right\}  .$ As is
explained above the theory of continuos family of the finite dimensional
operators [2, Chapter 2] can be applied to the part $\Omega(t)$ of the
spectrum of $T_{t}.$ Therefore, the unordered $m_{j}$-tuple $\Omega(t)$ depend
continuously (see page 108 of [2])\ on the parameter $t\in\left[  t_{0}%
-h_{k},t_{0}+h_{k}\right]  .$ Then by Theorem 5.2 of [2] (see page 109) there
exist $p$ single-valued continuous functions $\lambda_{1}(t),\lambda
_{2}(t),...,\lambda_{m_{j}}(t)$ the value of which constitute the $m_{j}%
$-tuple $\Omega(t)$ for $t\in\left[  t_{0}-h_{k},t_{0}+h_{k}\right]  .$
Moreover, it follows from (20) and Theorem 2 that
\begin{equation}
\operatorname{Re}\lambda_{l}(t_{0}-h_{k})<\lambda-\varepsilon_{k},\text{
}\operatorname{Re}\lambda_{l}(t_{0}+h_{k})>\lambda+\varepsilon_{k} \tag{21}%
\end{equation}
for $l=1,2,...,m_{j}.$ Now, we prove that $\lambda\in\left(  \cup_{l=1}%
^{m_{j}}\gamma_{l}\right)  ,$ where $\gamma_{l}$ is the curve $\left\{
\lambda_{l}(t):t\in\left[  t_{0}-h_{k},t_{0}+h_{k}\right]  \right\}  .$ Assume
the converse. Then by (21) the continuous curves $\gamma_{l}=\left\{
\lambda_{s}(t):t\in\left[  t_{0}-h_{k},t_{0}+h_{k}\right]  \right\}  $ extend
from $\lambda_{l}(t_{0}-h_{k})\ $to $\lambda_{l}(t_{0}+h_{k})$ pas above or
below of the point $\lambda$ for each $l=1,2,...,m_{j}.$ On the other hand, by
Theorem 1 if $\gamma_{l}$ passes above of $\lambda$ then there exist
$s\in\left\{  1,2,...,m_{j}\right\}  $ such that $s\neq l$ and $\gamma_{s}$
passes below of$\ \lambda.$ It implies that the number $m_{j}$\ of the curves
$\gamma_{1},\gamma_{2},...,\gamma_{m_{j}}$ is an even number. It contradicts
to the assumption that $m_{j}$ is an odd number. Thus, there exists $l$ such
that $\lambda\in\gamma_{l}.$ Since $\gamma_{l}\subset$ $\sigma(T),$ the
theorem is proved.
\end{proof}

Now we are ready to prove the main results of this paper. First let us
consider the Case 2

\begin{theorem}
Suppose that the matrix $C$ has a real eigenvalue $\mu_{j}$ of odd
multiplicity $m_{j}.$ If $n$ is odd number, then $\sigma(T)$ contains the set
$(-\infty,-H]\cup\lbrack H,\infty)$ for some $H\geq0.$
\end{theorem}

\begin{proof}
Let $\lambda$ be large real number. Without loss of generality assume that
$\lambda>\mu_{N,j}(-1),$ where $\mu_{k,j}(t)$ and $N$ are defined in (7) and
Theorem 2, respectively. Then there exists $t_{0}\in\lbrack-1,2\pi-1)$ such
that $\lambda=\mu_{k,j}(t_{0}))$\textit{ }for some $k\geq N$. It is clear that
there exists $h\in(0,\pi)$ such that $\left[  t_{0}-h_{k},t_{0}+h_{k}\right]
\subset\lbrack-h,2\pi-h).$ Therefore, it follows from Theorem 3$(a)$ that the
conditions of Lemma 1 holds. Then $\lambda\in\sigma(T)$ and the theorem is proved.
\end{proof}

Now, we study the complicated Case 3. Consider the intervals $U(i,j,k,\delta
_{k}+h_{k})$ and $U(i,j,-k-1,\delta_{k}+h_{k})$ for $i\in\left\{
1,2,...,s\right\}  ,$ where these intervals are obtained from the intervals
(16) and (17) by replacing $\delta_{k}$ with $\delta_{k}+h_{k}.$ By
definitions of $\delta_{k}$, $h_{k}$ and $\varepsilon_{k}$ we have
\begin{equation}
\delta_{k}+h_{k}=o(\left\vert k\right\vert ^{-1}). \tag{22}%
\end{equation}
On the other hand, it follows from (16) and (17) that there exists $c_{2}$
such that distance between the centres of these intervals is greater than
$c_{2}\left\vert k\right\vert ^{-1}.$ Therefore, if $\left\vert k\right\vert
\geq N$, then these intervals are pairwise disjoint and%
\[
\mu\left(
{\textstyle\bigcup\limits_{i\in\left\{  1,2,...,s\right\}  }}
\left(  U(i,j,k,\delta_{k}+h_{k})\cup U(i,j,-k-1,\delta_{k}+h_{k})\right)
\right)  =o(\left\vert k\right\vert ^{-1}),
\]
where $\mu(E)$ is the measure of the set $E.$ If we eliminate these intervals
from $[-1,2\pi-1)$ then the remaining set consists of the pairwise disjoint
intervals $[a_{1,k},b_{1,k}],[a_{2,k},b_{2,k}],...,[a_{l,k},b_{l,k}]$ such
that
\begin{equation}
\mu\left(  \lbrack a_{1,k},b_{1,k}]\cup\lbrack a_{2,k},b_{2,k}]\cup
...\cup\lbrack a_{l,k},b_{l,k}]\right)  =2\pi+o(\left\vert k\right\vert
^{-1}). \tag{23}%
\end{equation}

\begin{theorem}
Suppose that the matrix $C$ has a real eigenvalue $\mu_{j}$ of odd
multiplicity $m_{j}$ and $n$ is an even number. If $t_{0}\in\lbrack
a_{v,k},b_{v,k}]$ for some $v\in\left\{  1,2,...,l\right\}  $ and $k\geq N$,
then $\mu_{k,j}(t_{0})\in\sigma(T).$\textit{ Moreover, }%
\begin{equation}
\mu\left(  \mu_{k,j}([-1,2\pi-1))\cap\sigma(T)\right)  =\mu\left(  \mu
_{k,j}([-1,2\pi-1))\right)  (1+o(\left\vert k\right\vert ^{-1}).\tag{24}%
\end{equation}
In other words, the spectrum $\sigma(T)$ contains the set%
\begin{equation}
\mu_{k,j}\left(  [a_{1,k},b_{1,k}]\cup\lbrack a_{2,k},b_{2,k}]\cup
...\cup\lbrack a_{l,k},b_{l,k}]\right)  \tag{25}%
\end{equation}
for $k\geq N$ and $\sigma(T)$ contains the large part of the interval
$[0,\infty),$ in the sense that \
\begin{equation}
\lim_{\rho\rightarrow\infty}\frac{\mu([0,\rho]\backslash\sigma(T))}{\mu
(\sigma(T)\cap\lbrack0,\rho])}=0.\tag{26}%
\end{equation}

\end{theorem}

\begin{proof}
By the definition of the interval $[a_{v,k},b_{v,k}]$ if $t_{0}\in\lbrack
a_{v,k},b_{v,k}]$ then $t_{0}$ does not belong to any of the intervals
$U(i,j,k,\delta_{k}+h_{k})$ and $U(i,j,-k-1,\delta_{k}+h_{k}).$ It means that
$\left[  t_{0}-h_{k},t_{0}+h_{k}\right]  $ has no common points with the
intervals in (16) and (17). Therefore, using (18) we obtain that there exists
$h\in(0,\pi)$ such that$\left[  t_{0}-h_{k},t_{0}+h_{k}\right]  \subset
\lbrack-h,2\pi-h)\backslash U_{\delta_{k}}\left(  A(k,j)\right)  .$ Then by
Theorem 3$(b)$ the conditions of Lemma 1 holds. Thus $\mu_{k,j}(t_{0}%
)\in\sigma(T)$ for each $t_{0}\in\lbrack a_{v,k},b_{v,k}]$ and $v=1,2,...,l.$
It means that the set (25) belong to the spectrum. Therefore, using (7) one
can easily verify that (24) follows from (23) and (26) follows from (24).
\end{proof}

Now to prove the main result for Case 3 we use the following consequence of
Theorem 5.

\begin{corollary}
Suppose that the matrix $C$ has a real eigenvalue $\mu_{j}$ of odd
multiplicity $m_{j}$ and $n$ is an even number. Then \ there exist $H>0$ and
$\gamma_{k}=o(k^{n-2})$ such that the gaps of $\sigma(T)$ lying in
$[H,\infty)$ are contained in the union of the interval $S(2k,i,j)$ and
$S(2k+1,i,j)$ for $i=1,2,...,s$ and $\left\vert k\right\vert >N,$ where
\[
S(l,i,j)=\left(  \left(  \pi l\right)  ^{n}+\frac{\mu_{i}+\mu_{j}}{2}\left(
\pi l\right)  ^{n-2}-\gamma_{l},\left(  \pi l\right)  ^{n}+\frac{\mu_{i}%
+\mu_{j}}{2}\left(  \pi l\right)  ^{n-2}+\gamma_{l}\right)  .
\]

\end{corollary}

\begin{proof}
By definition, the set $[a_{1,k},b_{1,k}]\cup\lbrack a_{2,k},b_{2,k}%
]\cup...\cup\lbrack a_{l,k},b_{l,k}]$ is
\[
\lbrack-1,2\pi-1)\backslash\left(
{\textstyle\bigcup\limits_{i=1}^{s}}
\left(  U(i,j,k,\delta_{k}+h_{k})\cup U(i,j,-k-1,\delta_{k}+h_{k})\right)
\right)  .
\]
On the other hand, by Theorem 5 the image (25) of this set belong to
$\sigma(T).$ Therefore, using (7) one can easily conclude that there exists
$H>0$ such that the gaps of $\sigma(T)$ lying in $[H,\infty)$ is contained in%
\[
\mu_{k,j}\left(
{\textstyle\bigcup\limits_{i=1}^{s}}
\left(  U(i,j,k,\delta_{k}+h_{k})\cup U(i,j,-k-1,\delta_{k}+h_{k})\right)
\right)
\]
for $\left\vert k\right\vert >N.$ Moreover, using (7), (19) and (22) one can
easily verify that there exists $\gamma_{k}=o(k^{n-2})$ such that
\[
\mu_{k,j}\left(
{\textstyle\bigcup\limits_{i=1}^{s}}
U(i,j,k,\delta_{k}+h_{k})\right)  \subset S(2k,i,j)
\]
and
\[
\mu_{k,j}\left(
{\textstyle\bigcup\limits_{i=1}^{s}}
U(i,j,-k-1,\delta_{k}+h_{k})\right)  \subset S(2k+1,i,j).
\]
These inclusions give the proof of the corollary.
\end{proof}

Now we are ready to prove the main result of this paper for the Case 3.

\begin{theorem}
If the matrix $C$ has at least three real eigenvalues $\mu_{j_{1}},\mu_{j_{2}%
},\mu_{j_{3}}$ of odd multiplicity such that
\begin{equation}
\min_{i_{1},i_{2},i_{3}}\left(  diam(\{\mu_{j_{1}}+\mu_{i_{1}},\mu_{j_{2}}%
+\mu_{i_{2}},\mu_{j_{3}}+\mu_{i_{3}}\})\right)  \neq0, \tag{27}%
\end{equation}
where minimum is taken under condition $i_{j}\in\left\{  1,2,...,s\right\}  $
for $j=1,2,3$ and
\[
diam(E)=\sup_{x,y\in E}\mid x-y\mid,
\]
then there exists a number $H$ such that $[H,\infty)\subset\sigma(L)$.
\end{theorem}

\begin{proof}
By Corollary 1 the gaps lie in each of the following three sets
\[%
{\textstyle\bigcup\limits_{i=1,2,...,s;\left\vert k\right\vert >N}}
\left(  S(2k,i,j_{u})\cup S(2k+1,i,j_{u})\right)
\]
for $u=1,2,3.$ Therefore, to prove the theorem it is enough to show that
\ these sets have no common points. If they have a common point $x,$ then
using the definitions of these set we obtain that there exist $\left\vert
k\right\vert \geq N;$ $l\in\left\{  2k,2k+1\right\}  $ and $i_{u}\in\left\{
1,2,...,s\right\}  $ such that
\[
\mid x-(\pi l)^{n}-\frac{\mu_{j_{u}}+\mu_{i_{u}}}{2}\left(  \pi l\right)
^{n-2}\mid<\beta_{l}%
\]
for all $u=1,2,3$, where $\beta_{l}=o(l^{n-2}).$ This inequality implies that
$\mu_{j_{1}}+\mu_{i_{1}}=\mu_{j_{2}}+\mu_{i_{2}}=\mu_{j_{3}}+\mu_{i_{3}}$
which contradicts (27). The theorem is proved.
\end{proof}

\begin{remark}
Using (6) and Theorem 1 we obtain

\textbf{(a) }If $m$ is an odd number, then the matrix $C$ has a real
eigenvalues of odd multiplicity.

\textbf{(b) }If $m$ is an even number, then the number of real eigenvalues of
odd multiplicity of the matrix $C$ is an even number.
\end{remark}

\section{Appendix}

In this section we give the proof of Theorem 2. For this we prove that if
$\lambda(t,\varepsilon)$ is an eigenvalue of the operator $T_{t}%
(\varepsilon,C)$ satisfying the inequality
\begin{equation}
|\lambda(t,\varepsilon)-\mu_{k,j}(t)|\leq|\lambda(t,\varepsilon)-\mu
_{d,u}(t)|\tag{28}%
\end{equation}
for all $\left(  d,u\right)  \neq(k,j)$, then
\begin{equation}
\lambda(t,\varepsilon)\in U_{\varepsilon_{k}}(\mu_{k,j}(t)),\tag{29}%
\end{equation}
where $\left\vert k\right\vert \geq N$ and $N$ is defined in Theorem 2. We use
the following notations and formulas. Let $A(k,t)$ be $\left\{  k\right\}  ,$
if $n$ is an odd number. When $n$ is an even number, then let $A(k,t)$ be
$\left\{  k\right\}  ,$ $\left\{  \pm k\right\}  $ and $\left\{
k,-k-1\right\}  $ respectively, if \ $t\in\left(  \lbrack1,\pi-1]\cup
\lbrack\pi+1,2\pi-1)\right)  ,$ $t\in\lbrack-1,1)$ and $t\in(\pi-1,\pi+1).$
Using the obvious inequality $\left\vert a^{n}-b^{n}\right\vert \geq\left\vert
a-b\right\vert \left\vert a^{n-1}+b^{n-1}\right\vert $ for $a>0$ and $b>0,$
one can easily verify that if $\left\vert k\right\vert \geq N$ and $d\notin
A(k,t),$ then
\[
\left\vert \left(  2\pi k+t\right)  ^{n}-\left(  2\pi d+t\right)
^{n}\right\vert \geq\pi\left\vert \left\vert k\right\vert -\left\vert
d\right\vert \right\vert \left(  \left\vert k\right\vert ^{n-1}+\left\vert
d\right\vert ^{n-1}\right)  .
\]
Therefore it follows from (7) and (28) we obtain
\begin{equation}
\left\vert \lambda(t,\varepsilon)-\left(  2\pi d+t\right)  ^{n}\right\vert
>\left\vert \left\vert k\right\vert -\left\vert d\right\vert \right\vert
\left(  \left\vert k\right\vert ^{n-1}+\left\vert d\right\vert ^{n-1}\right)
\tag{30}%
\end{equation}
for all $d\notin A(k,t),$ $t\in\lbrack-1,2\pi-1)$\textit{ }and $\varepsilon
\in\lbrack0,1].$

The eigenfunctions $\Phi_{d,u,l,t}^{\ast}(x)$ and associated functions
$\Phi_{d,u,l,q,t}^{\ast}(x)$ of $\left(  T_{t}(C)\right)  ^{\ast}$
corresponding to the eigenvalue $\overline{\mu_{d,u}(t)}$ are
\begin{equation}
\Phi_{d,u,l,t}^{\ast}(x)=v_{u,l}^{\ast}e^{i\left(  2\pi d+t\right)  x},\text{
}\Phi_{d,u,l,q,t}^{\ast}(x)=v_{u,l,q}^{\ast}e^{i\left(  2\pi d+t\right)  x},
\tag{31}%
\end{equation}
where $v_{u,l}^{\ast}$ and $v_{u,l,q}^{\ast}$ are the eigenvector and
associated vector of \ $C^{\ast}$ corresponding to $\overline{\mu_{u}},$
$l=1,2,...,l_{u\text{ }}$ and $q=1,2,...,r_{u,l}-1$ (see (8) and (10)). In the
other words,
\begin{equation}
(\left(  T_{t}(C)\right)  ^{\ast}-\overline{\mu_{d,u}(t)}I)\Phi_{d,u,l,t}%
^{\ast}=0 \tag{32}%
\end{equation}
and
\begin{equation}
\text{ }(\left(  T_{t}(C)\right)  ^{\ast}-\overline{\mu_{d,u}(t)}%
I)\Phi_{d,u,l,q,t}^{\ast}=\Phi_{d,u,l,q-1,t}^{\ast}, \tag{33}%
\end{equation}
where $\Phi_{d,u,l,0,t}^{\ast}(x)=\Phi_{d,u,l,t}^{\ast}(x)$. Let
$\Psi_{\lambda(t,\varepsilon)}$ be a normalized eigenfunction of
$T_{t}(\varepsilon,C)$ corresponding to the eigenvalue $\lambda(t,\varepsilon
).$ Multiplying both sides of
\begin{equation}
T_{t}(\varepsilon,C)\Psi_{\lambda(t,\varepsilon)}=\lambda(t,\varepsilon
)\Psi_{\lambda(t,\varepsilon)} \tag{34}%
\end{equation}
by $\Phi_{d,u,l,t}^{\ast}(x),$ using $T_{t}(\varepsilon,C)=T_{t}%
(C)+(T_{t}(\varepsilon,C)-T_{t}(C))$ and (32), we get%
\[
(\lambda(t,\varepsilon)-\mu_{d,u}(t))(\Psi_{\lambda(t,\varepsilon)}%
,\Phi_{d,u,l,t}^{\ast})=(T_{t}(\varepsilon,C)-T_{t}(C))\Psi_{\lambda
(t,\varepsilon)},\Phi_{d,u,l,t}^{\ast}).
\]
Similarly, multiplying (34) by $\Phi_{d,u,l,1,t}^{\ast}$ and using (33) for
$q=1$ we obtain
\[
\left(  \lambda(t,\varepsilon)-\mu_{d,u}(t)\right)  (\Psi_{\lambda
(t,\varepsilon)},\Phi_{d,u,l,1,t}^{\ast})=
\]%
\[
(T_{t}(\varepsilon,C)-T_{t}(C))\Psi_{\lambda(t,\varepsilon)},\Phi
_{d,u,l,1,t}^{\ast})+(\Psi_{\lambda(t,\varepsilon)},\Phi_{d,u,l,t}^{\ast}).
\]
Now, using the last two equalities, one can easily verify that%
\[
(\lambda(t,\varepsilon)-\mu_{d,u})^{2}(\Psi_{\lambda(t,\varepsilon)}%
,\Phi_{d,u,l,1,t}^{\ast})=
\]%
\[
(\lambda(t,\varepsilon)-\mu_{d,u})((T_{t}(\varepsilon,C)-T_{t}(C))\Psi
_{\lambda(t,\varepsilon)},\Phi_{d,u,l,1,t}^{\ast})+
\]%
\[
((T_{t}(\varepsilon,C)-T_{t}(C))\Psi_{\lambda(t,\varepsilon)},\Phi
_{d,u,l,t}^{\ast}).
\]
In this way one can deduce the formulas
\begin{equation}
(\lambda(t,\varepsilon)-\mu_{d,u}(t))^{q+1}(\Psi_{\lambda(t,\varepsilon)}%
,\Phi_{d,u,l,q,t}^{\ast})= \tag{35}%
\end{equation}%
\[
\sum_{p=0}^{q}(\lambda(t,\varepsilon)-\mu_{d,u})^{p}((T_{t}(\varepsilon
,C)-T_{t}(C))\Psi_{\lambda(t,\varepsilon)},\Phi_{d,u,l,p,t}^{\ast}).
\]

To prove (29) we estimate the terms of (35).

\begin{lemma}
If $n$ is an even number, then there exists $d\in A(k,t)$ such that
\begin{equation}
\left\vert (\Psi_{\lambda(t,\varepsilon)},\Phi_{d,u,l,q,t}^{\ast})\right\vert
\geq c_{3} \tag{36}%
\end{equation}
for some $u,l,q$ and
\begin{equation}
\left\vert ((T_{t}(\varepsilon,C)-T_{t}(C))\Psi_{\lambda(t,\varepsilon)}%
,\Phi_{d,u,l,p,t}^{\ast})\right\vert \leq c_{4}\left(  \frac{1}{|k|}%
+q_{k}\right)  \left\vert k\right\vert ^{n-2} \tag{37}%
\end{equation}
for all $u,l,p$, where $\left\vert k\right\vert \geq N$, $N$ and $q_{k}$ are
defined in Theorem 2. If $n$ is an odd number, then (36) and (37) hold for
$d=k$ and $q_{k}=0.$
\end{lemma}

\begin{proof}
To prove the lemma we use the following formula
\begin{equation}
\left(  \lambda(t,\varepsilon)-\left(  2\pi d+t\right)  ^{n}\right)  \left(
\Psi_{\lambda(t,\varepsilon)},\varphi_{d,s,t}\right)  =%
{\textstyle\sum\limits_{\nu=2}^{n}}
(-i)^{n-v}(P_{\nu}\Psi_{\lambda(t,\varepsilon)}^{\left(  n-\nu\right)
},\varphi_{d,s,t})\tag{38}%
\end{equation}
which can be obtained from (34) by multiplying by $\varphi_{d,s,t}%
(x)=:e_{s}e^{i\left(  2\pi d+t\right)  x}$ and using the equality
$T_{t}\left(  0\right)  \varphi_{d,s,t}=\left(  2\pi d+t\right)  ^{n}%
\varphi_{d,s,t},$ where $e_{1},e_{2},...,e_{m}$ is a standard basis of
$\mathbb{C}^{m},$ $T_{t}(0)$ is the operator generated by the expression
$(-i)^{n}y^{(n)}$ and the boundary conditions (3). Using (38), (30) and Bessel
inequality for the orthonormal system
\begin{equation}
\left\{  \varphi_{d,s,t}(x)=:e_{s}e^{i\left(  2\pi d+t\right)  x}%
:d\in\mathbb{Z},\text{ }s=1,2,...,m\right\}  \tag{39}%
\end{equation}
we obtain
\[
\sum\limits_{d\in\left(  \mathbb{Z}\backslash A(k,t)\right)  ,\text{
}s=1,2,...,m}\left\vert \left(  \Psi_{\lambda(t,\varepsilon)},\varphi
_{d,s,t}\right)  \right\vert ^{2}\leq\frac{\left\Vert
{\textstyle\sum\limits_{\nu=2}^{n}}
(-i)^{n-v}P_{\nu}\Psi_{\lambda(t,\varepsilon)}^{\left(  n-\nu\right)
}\right\Vert ^{2}}{k^{2n-2}}.
\]
On the other hand, in [9] (see Lemma 2 of [9]) we proved that there exists
$c_{5}$ such that
\begin{equation}
\left\vert \Psi_{\lambda(t,\varepsilon)}^{(\nu)}(x)\right\vert \leq
c_{5}\left\vert k\right\vert ^{\nu}\tag{40}%
\end{equation}
for all $x\in\lbrack0,1],t\in\lbrack-1,2\pi-1)$ and $\varepsilon\in
\lbrack0,1].$ Therefore we have
\[
\sum\limits_{d\in\left(  \mathbb{Z}\backslash A(k,t)\right)  ,s=1,2,...,m}%
\left\vert \left(  \Psi_{\lambda(t,\varepsilon)},\varphi_{d,s,t}\right)
\right\vert ^{2}\leq\frac{c_{6}}{k^{2}}.
\]
Then by the Parsevals equality
\[
\sum\limits_{d\in A(k,t),s=1,2,...,m}\left\vert \left(  \Psi_{\lambda
(t,\varepsilon)},\varphi_{d,s,t}\right)  \right\vert ^{2}\geq\frac{1}{2}.
\]
Since the system of the root vectors of the matrix $C^{\ast}$ is a basis of
$\mathbb{C}^{m}$, (31) and the last inequality imply that (36) holds for some
$u,l,q$ and $d\in A(k,t).$ If $n$ is an odd number, then $A(k,t)=\left\{
k\right\}  $ and hence (36) holds for $d=k.$

Now we prove (37). By the definitions of $T_{t}(\varepsilon,C)$ and $T_{t}(C)$
we have
\begin{align}
((T_{t}(\varepsilon,C)-T_{t}(C))\Psi_{\lambda(t,\varepsilon)},\Phi
_{d,u,l,p,t}^{\ast}) &  =((P_{2}-C)\Psi_{\lambda(t,\varepsilon)}^{(n-2)}%
,\Phi_{d,u,l,p,t}^{\ast})+\tag{41}\\
&
{\textstyle\sum\limits_{v=3}^{n}}
(P_{v}\Psi_{\lambda(t,\varepsilon)}^{\left(  n-\nu\right)  },\Phi
_{d,u,l,p,t}^{\ast}).\nonumber
\end{align}
Using (40) and (31) we obtain that
\[
\left\vert
{\textstyle\sum\limits_{v=3}^{n}}
(P_{v}\Psi_{\lambda(t,\varepsilon)}^{\left(  n-\nu\right)  },\Phi
_{d,u,l,p,t}^{\ast})\right\vert \leq c_{7}\left\vert k\right\vert ^{n-3}.
\]
By (31) to estimate the first term in the right side of (41) it is enough to
prove that
\[
\left\vert ((P_{2}-C)\Psi_{\lambda(t,\varepsilon)}^{(n-2)},\varphi
_{d,i,t})\right\vert \leq c_{8}q_{k}\left\vert k\right\vert ^{n-2}%
\]
for $i=1,2,...,m.$ Using the decomposition
\[
\Psi_{\lambda(t,\varepsilon)}^{(n-2)}=\sum\limits_{l\in\mathbb{Z},\text{
}s=1,2,...,m}\left(  \Psi_{\lambda(t,\varepsilon)}^{(n-2)},\varphi
_{l,s,t}\right)  \varphi_{l,s,t}%
\]
of $\Psi_{\lambda(t,\varepsilon)}^{(n-2)}$ by the orthonormal basis (39) we
obtain.
\begin{equation}
((P_{2}-C)\Psi_{\lambda(t,\varepsilon)}^{(n-2)},\varphi_{d,i,t})=\sum
\limits_{l\in\left(  \mathbb{Z}\backslash\left\{  d\right\}  \right)  ,\text{
}s=1,2,...,m}p_{2,i,s,d-l}\left(  \Psi_{\lambda(t,\varepsilon)}^{(n-2)}%
,\varphi_{l,s,t}\right)  .\tag{42}%
\end{equation}
The right-hand side of (42) is the sum of
\[
S_{1}=:\sum_{l\in A(k,t)\backslash\left\{  d\right\}  ;\text{ }s=1,2,...,m}%
p_{2,i,s,d-l}\left(  \Psi_{\lambda(t,\varepsilon)}^{(n-2)},\varphi
_{l,s,t}\right)
\]
and%
\[
S_{2}=:\sum\limits_{l\in\left(  \mathbb{Z}\backslash A(k,t)\right)  ;\text{
}s=1,2,...,m}p_{2,i,s,d-l}\left(  \Psi_{\lambda(t,\varepsilon)}^{(n-2)}%
,\varphi_{l,s,t}\right)  .
\]
First, let us estimate $S_{1}.$ It follows from the definition of $A(k,t)$
that the set $A(k,t)\backslash\left\{  d\right\}  $ for $d\in A(k,t)$ consists
of at most one number. Moreover, it follows from the definition of $A(k,t)$
that $(d-l)\in\left\{  \pm2k,\pm(2k+1),\right\}  $ for all $d\in A(k,t)$ and
$l\in\left(  A(k,t)\backslash\left\{  d\right\}  \right)  .$ Therefore, using
(40) and the definition of $q_{k}$ we obtain
\begin{equation}
\left\vert S_{1}\right\vert \leq c_{9}q_{k}\left\vert k\right\vert
^{n-2}.\tag{43}%
\end{equation}
It remains to estimate $S_{2}.$ Using the Schwards inequality for the space
$l_{2}$ and the integration by parts formula and the last relation in (2) we
obtain
\begin{equation}
\left\vert S_{2}\right\vert ^{2}\leq c_{10}\sum\limits_{l\in\left(
\mathbb{Z}\backslash A(k,t)\right)  ;\text{ }s=1,2,...,m}|l|^{2n-4}\left\vert
\left(  \Psi_{\lambda(t,\varepsilon)}^{{}},\varphi_{l,s,t}\right)  \right\vert
^{2}.\tag{44}%
\end{equation}
Now, first use (38) and then (30)\ in (44) to conclude that%
\[
\left\vert S_{2}\right\vert ^{2}\leq\sum\limits_{\substack{l\in\left(
\mathbb{Z}\backslash A(k,t)\right)  ,\\s=1,2,...,m}}\frac{c_{11}|l|^{2n-4}%
}{\left\vert \left\vert k\right\vert -\left\vert l\right\vert \right\vert
^{2}\left(  \left\vert k\right\vert ^{n-1}+\left\vert l\right\vert
^{n-1}\right)  ^{2}}\left\vert
{\textstyle\sum\limits_{v=2}^{n}}
(-i)^{n-v}(P_{v}\Psi_{\lambda(t,\varepsilon)}^{\left(  n-v\right)  }%
,\varphi_{l,s,t})\right\vert ^{2}.
\]
Finally, using (40) in the last inequality we get
\[
\left\vert S_{2}\right\vert ^{2}\leq\sum\limits_{l\in\left(  \mathbb{Z}%
\backslash A(k,t)\right)  ,\text{ }s=1,2,...,m}\frac{c_{12}|l|^{2n-4}%
|k|^{2n-4}}{\left\vert \left\vert k\right\vert -\left\vert l\right\vert
\right\vert ^{2}\left(  \left\vert k\right\vert ^{n-1}+\left\vert l\right\vert
^{n-1}\right)  ^{2}}%
\]
from which by direct calculations we obtain
\begin{equation}
\left\vert S_{2}\right\vert ^{2}\leq c_{13}|k|^{2n-6},\text{ }\left\vert
S_{2}\right\vert \leq c_{14}|k|^{n-3}.\tag{45}%
\end{equation}
Thus (37) follows from (43) and (45). It is clear that if $n$ is an odd
number, then $A(k,t)\backslash\left\{  d\right\}  $ is an empty set for $d\in
A(k,t).$ Therefore $S_{1}=0$ and in (37) the term $q_{k}$ does not appear.
\end{proof}

Now we are ready to prove Theorem 2. Dividing (35) by $(\Psi_{\lambda
(t,\varepsilon)},\Phi_{d,u,l,q,t}^{\ast}),$ and using (36) and (37) we obtain
\[
(\lambda(t,\varepsilon)-\mu_{d,u}(t))^{q+1}=\sum_{p=0}^{q}(\lambda
(t,\varepsilon)-\mu_{d,u}(t))^{p}\left\vert k\right\vert ^{n-2}O\left(
\frac{1}{|k|}+q_{k}\right)  ,
\]
where $q+1\leq r$ and $r$ is defined in Theorem 2. From the last equality we
obtain that $\lambda(t,\varepsilon)\in U_{\varepsilon_{k}}(\mu_{d,u}(t)).$
This inclusion with (28) implies (29) which gives the proof of Theorem 2.


\begin{thebibliography}{9}                                                                                                %


\bibitem {}F. Bagarello, J. P. Gazeau, F. H. Szafraniec and M. Znojil,
\textit{Non-self-adjoint Operators in Quantum Physics: Mathematical Aspects}
(John Wiley \& Sons, Inc. Published, 2015).

\bibitem {}T. Kato, \textit{Perturbation Theory for Linear Operators}
(Springer-Verlag, Berlin, 1980).

\bibitem {}D. C. McGarvey, "Differential operators with periodic coefficients
in $L_{p}(-\infty,\infty)"$, \textit{J. Math. Anal. Appl.} \textbf{11},
564-596 (1965).

\bibitem {}M. A. Naimark,\textit{ Linear Differential Operators} (George G.
Harap\&Company, London, 1967).

\bibitem {}F. S. Rofe-Beketov, "The spectrum of nonselfadjoint differential
operators with periodic coefficients", \textit{Soviet Math. Dokl.} \textbf{4},
1563-1566 (1963).

\bibitem {}O. A. Veliev, \textit{Non-self-adjoint Schr\"{o}dinger operator
with a periodic potential}, (Springer Nature, Switzerland, 2021).

\bibitem {}O. A. Veliev, "On the Schr\"{o}dinger operator with a Periodic
PT-symmetric Matrix Potential," J. Math. Phys. \textbf{62}, 103501 (2021).

\bibitem {}O. A. Veliev, "On the differential operators of odd order with
PT-symmetric periodic matrix coefficients", arXiv:2303.08703.

\bibitem {}O. A. Veliev, "On the Differential Operators with Periodic Matrix
Coefficients, Abstract and Applied Analysis" 2009; ID 934905: 1-21. https://doi.org/10.1155/2009/934905.
\end{thebibliography}
\end{document}